\documentclass[11pt,a4paper]{article}

\usepackage{amsmath}
\usepackage{amsthm}
\usepackage{amssymb}
\usepackage{amsfonts}
\usepackage{enumerate}

\usepackage{dsfont}

\theoremstyle{definition}
\newtheorem{theorem}{Theorem}%[section] links numbering to section
\newtheorem{lemma}[theorem]{Lemma}  %[theorem] between brackets links numbering to theorems and for props etc.

\newtheorem{prop}[theorem]{Proposition}
\newtheorem{corollary}[theorem]{Corollary}
\newtheorem{remark}[theorem]{Remark}

\numberwithin{equation}{section}
\numberwithin{theorem}{section}

\newcommand{\R}{\ensuremath{\mathbb{R}}}
\newcommand{\Z}{\ensuremath{\mathbb{Z}}}
\newcommand{\N}{\ensuremath{\mathbb{N}}}

\newcommand{\ER}{Erd\H{o}s-R\'enyi }

\begin{document}

\title
{A moment-generating formula for \ER component sizes}

\author { Bal\'azs R\'ath\textsuperscript{1}}

%\footnotetext[1]{University of Oxford, United Kingdom. E-mail: martin@stats.ox.ac.uk}
\footnotetext[1]{MTA-BME Stochastics Research Group, Budapest University of Technology and Economics, Hungary. E-mail: rathb@math.bme.hu}

\maketitle

\begin{abstract}
We derive a simple formula characterizing the distribution of the size
of the connected component of a fixed vertex in the \ER random graph which allows
us to give elementary proofs of some results of 
 \cite{tim_et_al} and \cite{janson_luczak} about the susceptibility  in the subcritical graph and the CLT \cite{pittel} for the size of the giant component in the supercritical graph.

\bigskip

\medskip

\end{abstract}

\section{Introduction}
\label{section_intro}

The \ER graph $\mathcal{G}_{n,p}$, introduced in \cite{erdos_renyi}, is the random graph on $n$ vertices where each pair of vertices
is connected with probability $p$, independently from each other. For an introduction
to this fundamental mathematical model of large networks, see \cite{bollobas_rg_book, janson_lu_ru_book, remco_book}.

We denote by $\mathbb{P}_{n,p}$ the law of $\mathcal{G}_{n,p}$ and $\mathbb{E}_{n,p}$ the corresponding expectation. 

We assume that the vertex set of $\mathcal{G}_{n,p}$ is $[n]=\{1,\dots,n\}$ and we denote by
$\mathcal{C}$ the connected component in $\mathcal{G}_{n,p}$ of the vertex indexed by $1$.
We denote by $|\mathcal{C}|$ the number of vertices  of $\mathcal{C}$.

For any $n \in \N$, $p \in [0,1]$,  $j \in \Z \cap (-n, +\infty)$,  and $k \in [n]$ we define
\begin{equation}\label{g_n_p_k}
g_{n,p}(j,k)= (1-p)^{jk} \prod_{i=0}^{k-1}   \frac{n-i+j}{ n -i } .
\end{equation}

The central result of this short note is the following  formula:

\begin{prop}\label{prop_ER_mgf_general} For any $n \in \N$, $j \in \Z \cap (-n, +\infty)$ and $p \in [0,1]$ we have
\begin{equation}\label{mgf_ER_component_positive}
\mathbb{E}_{n,p} \left[\, g_{n,p}(j,|\mathcal{C}| )
  \, \right]=\frac{n+j}{n}\left(1-\mathbb{P}_{n+j,p}[ \, |\mathcal{C}|>n \, ] \right).
\end{equation}
\end{prop}
Note that if $j\leq 0$ then the r.h.s.\ is simply $\frac{n+j}{n}$. 
We prove Proposition \ref{prop_ER_mgf_general} in Section \ref{section_proof_of_combinatorial_prop}.

\begin{remark}\label{remark_unique_characterization} 
\begin{enumerate}[(i)]
\item \label{rem_unique}  Define the $n \times n$ matrix  $M$ by $M_{j,k}=g_{n,p}(j, k)$
for  $j \in \Z \cap (-n, 0]$ and $k \in [n]$. 
The matrix $M$ is triangular with non-zero diagonal entries, hence it is invertible. Therefore, Proposition \ref{prop_ER_mgf_general} uniquely characterizes the distribution of
$|\mathcal{C}|$ under $\mathbb{P}_{n,p}$.
\item A generalization of Proposition \ref{prop_ER_mgf_general} appears in Proposition 1.6 of the recent
  preprint \cite{lemaire}, see also \cite[Remark 1.7]{lemaire}. The random graph process studied in \cite{lemaire} can be informally defined
  as follows: starting from the empty graph on the vertex set $[n]$, cliques are added with a rate that only depends on their size (the dynamical \ER graph is the special case when only cliques of size two are added).
\end{enumerate}
\end{remark}

Proposition \ref{prop_ER_mgf_general} allows us to give short and self-contained proofs of some
delicate results about the sizes of connected components of the \ER graph in the subcritical
(see Theorem \ref{thm_subcrit}) as well as the supercritical
(see Theorem \ref{thm_CLT_giant}) cases. First, we give a short non-rigorous demonstration of
how our formula is used in Remark \ref{remark_genereting_fn_borel}.

When we study the phase transition of the \ER graph, it is natural to introduce a parameter $t \in \R_+$
and to study $\mathcal{G}_{n,p}$ for
\begin{equation}\label{p_t_n_def}
p=p(t,n)=1-e^{-t/n}.
\end{equation}
We will fix this relation between $p$ and $t$ throughout this paper.

For any $n \in \N$,  $\lambda \in \R$,  and $k \in [n]$ we define
\begin{equation}\label{f_n_lambda_k}
f_{n,t}(\lambda,k)=\prod_{i=0}^{k-1} 
e^{-\lambda t} \cdot \left( 1+ \frac{\lambda}{ 1 -\frac{i}{n} } \right),
\end{equation}
so that we have $f_{n,t}(\frac{j}{n},k)=g_{n,p}(j,k)$ if $j \in \Z \cap (-n, +\infty)$ and thus
\begin{equation}\label{mgf_ER_positive_f}
\mathbb{E}_{n,p} \left[\, f_{n,t}(\lambda,|\mathcal{C}| )
  \, \right] \stackrel{ \eqref{mgf_ER_component_positive} }{=}(1+\lambda)\left(1-\mathbb{P}_{(1+\lambda)n,p}[ \, |\mathcal{C}|>n \, ] \right),
  \quad \lambda \in \frac{\Z}{n} \cap (-1,+\infty).
\end{equation}

\begin{remark}\label{remark_genereting_fn_borel}
If we  fix  $t < 1$ and 
(non-rigorously) denote   $G_t(z)=\lim_{n \to \infty} \mathbb{E}_{n,p(t,n)}[z^{ |\mathcal{C}|}]$ for any $z \in [0,1]$,  
then for $\lambda =z-1$  we (non-rigorously) obtain
\begin{equation}\label{borel_mgf}
z=1+\lambda \stackrel{ \eqref{mgf_ER_positive_f} }{=}
\lim_{n \to \infty} \mathbb{E}_{n, p(t,n)} \left[ f_{n,t}(\lambda,|\mathcal{C}|)  \right]
\stackrel{ \eqref{f_n_lambda_k} }{=} G_t\left( e^{-\lambda t} \cdot \left( 1+ \lambda \right) \right)=
   G_t\left( e^{(1-z)t } z \right).
\end{equation}
Thus $G_t(z)=-W(-e^{-t}tz)/t$, where $W$ is the Lambert-W function.
Now it is known that if $p=1-e^{-t/n}$ and  $n \to \infty$ then  $|\mathcal{C}|$ converges in distribution to the total number of offspring in a subcritical Galton-Watson branching process
with $\mathrm{POI}(t)$ offspring distribution (see \cite[Theorem 11.6.1]{alon_spencer_fourth}), i.e., 
 $|\mathcal{C}|$ has Borel distribution with parameter $t$ (see \cite[Section 2.2]{aldous_smol} or \cite[Section 7]{janson_luczak}). 
 The generating function $G_t$
of the Borel distribution with parameter $t$ is known to be characterized by
the identity $G_t(z)\equiv z e^{(G_t(z)-1)t}$ (see \cite[Section 10.4]{alon_spencer_second}), which is in turn equivalent to $G_t(z)=-W(-e^{-t}tz)/t$, therefore a more rigorous version of 
\eqref{borel_mgf} can be used to show that the distribution $|\mathcal{C}|$ weakly converges to the
Borel distribution with parameter $t$ as $n \to \infty$.
\end{remark}

Now we state our rigorous results.
We will use the Bachmann-Landau big O notation: we write $f(n,t)=\mathcal{O}\left(g(n,t)\right)$
if there exists a universal constant $C$ such that $f(n,t) \leq C g(n,t)$ for any $n \in \N$ and
any $t$ in an explicitly specified domain. We write $f(n)=\mathcal{O}\left( g(n) \right)$ if there
exists a constant $C$ (that may depend on $t$) such that $f(n) \leq C g(n)$ for any $n \in \N$.

We will give a short and self-contained proof of some results of  \cite{tim_et_al} and \cite{janson_luczak}:

\begin{theorem}\label{thm_subcrit} For any $t \in [0, 1-n^{-1/3}]$ we have
\begin{align}
\label{tim_error_susceptibility}
\mathbb{E}_{n,p}(|\mathcal{C}|)&=\frac{1}{1-t}+
\frac{\frac{t^2}{2}-t }{(1-t)^4  }\frac{1}{n} +
\mathcal{O}\left( \frac{1}{(1-t)^7 }\frac{1}{n^2}  \right), \\
\label{janson_second_moment_error}
\mathbb{E}_{n,p}(|\mathcal{C}|^2)
&= \frac{1}{(1-t)^3}+ \mathcal{O} \left(\frac{1}{(1-t)^6 }\frac{1}{n} \right).
\end{align}
\end{theorem}

We will prove  Theorem \ref{thm_subcrit} in Section \ref{section:subcrit}.

\begin{remark}\label{remark_janson_tim} $\mathbb{E}_{n,p}(|\mathcal{C}|)$ is often called
the \emph{susceptibility} of the \ER graph.

\begin{enumerate}[(i)]
\item Equation (1.15) of \cite[Theorem 1.2]{tim_et_al} states that if 
$p=\frac{\mu}{n-1}$ and $0<\mu<1$ then
\begin{equation}\label{true_tim}
\mathbb{E}_{n,p}(|\mathcal{C}|)=\frac{1}{1-\mu}-\frac{2 \mu^2-\mu^4}{2(1-\mu)^4}\frac{1}{n} +
\mathcal{O}\left( \frac{1}{n^2} \right).
\end{equation}
Now  \eqref{true_tim} follows from \eqref{tim_error_susceptibility}
if we take into account that $\mu=(n-1)(1-e^{-t/n})$. The proof of
\eqref{true_tim} in \cite[Section 2]{tim_et_al} uses a coupling of the breadth-first exploration process
of $\mathcal{C}$ and a process related to a branching random walk. Our proof of \eqref{tim_error_susceptibility}
is completely different as it only uses Proposition \ref{prop_ER_mgf_general}.

\item Equation (1.3) of \cite[Theorem 1.1]{janson_luczak} follows from our 
\eqref{tim_error_susceptibility}. In fact it already follows from our short Lemma \ref{lemma_third_order_formula_minus_one}, see \eqref{first_order_expansion_of_first_moment}.
Our \eqref{janson_second_moment_error} is equivalent to one of the statements about $S_3$ in
\cite[Theorem 3.4]{janson_luczak}. The proofs of these results in \cite[Section 3]{janson_luczak} use differential equations (in the variable $t$) and are completely different from ours.
\item Both statements of Theorem \ref{thm_subcrit} give something meaningful 
in the whole subcritical regime outside the critical window, e.g., 
the first term of the r.h.s.\ of \eqref{tim_error_susceptibility} is much bigger than the second one, which is
much bigger than the third one if $(1-t)^3 n \gg 1$.
\end{enumerate}
\end{remark}

We also give a short and self-contained proof of the central limit theorem proved in \cite{pittel} for the size of the giant connected
component of $\mathcal{G}_{n,p}$ (see also 
 \cite{bbv}, \cite{remco_book} and \cite{martin_lof_clt} for alternative proofs). Our proof only uses Proposition \ref{prop_ER_mgf_general}, see Theorem \ref{thm_CLT_giant} below. We begin with some notation.

Given some $t>1$ let us define the function $\varphi: [0,1) \to \R$ by
\begin{equation}\label{varphi_def}
\varphi(x)=-xt-\ln(1-x).
\end{equation}
Then $\varphi$ is a convex function satisfying $\varphi(0)=0$, $\varphi'(0)<0$ and $\varphi(1_-)=+\infty$.

Given $t>1$ define $\theta=\theta(t)\in (0,1)$ to be unique number for which 
\begin{equation}\label{theta_def}
\varphi(\theta)=0, \quad \text{ or, equivalently } \quad
e^{t\theta}(1-\theta)=1.
\end{equation}
Note that $\theta(t)$ is the survival probability of a branching process with $\mathrm{POI}(t)$ offspring distribution (however, our proof of Theorem \ref{thm_CLT_giant} below does not make use of this fact).

We also  note that it follows from $\varphi(0)=0$, $\varphi'(0)<0$ and $\varphi''(x)>0, \, x \in [0,1)$ that
\begin{equation}\label{varphi_prime_theta}
0<\varphi'(\theta)=-t+\frac{1}{1-\theta}.
\end{equation}

Recall the notion of $p=p(t,n)=1-e^{-t/n}$ from \eqref{p_t_n_def}. 
\begin{theorem}\label{thm_CLT_giant} 
 Let us denote by $|\mathcal{C}_{max}|$ the size of the largest connected component of $\mathcal{G}_{n,p}$.
 For any $t>1$ we have
\begin{equation}\label{clt_giant_formula}
\lim_{n \to \infty} \mathbb{P}_{n, p(t,n)}\left[ \frac{ |\mathcal{C}_{max}|-\theta n }{\sigma \sqrt{n}} \leq x  \right]=
\Phi(x), \quad \text{where} \quad
\sigma= \frac{\sqrt{\theta}}{\varphi'(\theta)\sqrt{1-\theta}} 
\end{equation}
and $\Phi(x)$ is the c.d.f.\ of the standard normal distribution.
\end{theorem}
We prove Theorem \ref{thm_CLT_giant} in Section \ref{section:clt}. 
Our proof is  different from earlier proofs, which use the joint CLT for tree components of various sizes \cite{pittel},
 stochastic differential equations which arise in the context of epidemics
\cite{martin_lof_clt}, and exploration processes \cite{bbv, remco_book}.

\begin{remark}
We believe that Proposition \ref{prop_ER_mgf_general}  can also be used to give elementary alternative proofs of some results of \cite{aldous_coalescent} on the sizes of connected components in the critical \ER graph.
In particular, let $\rho^u_0$ denote the sigma-finite excursion length measure of the ``first'' excursion of the 
Brownian motion with parabolic drift which encodes the block sizes of the standard multiplicative coalescent process at time $u \in \mathbb{R}$ (see \cite[(64)]{aldous_coalescent}). We believe that if
$t:=1+u n^{-1/3}$  and $X_n:=|\mathcal{C}|/n^{2/3}$ then the formula
\begin{equation}\label{exc_length_limit_comp_size}
\lim_{n \to \infty} n^{1/3} \mathbb{P}_{n, p(t,n)}[\, X_n > x \, ]=\rho^u_0(x,+\infty), \qquad x \in (0,+\infty)
\end{equation}
can be proved using the methods of this paper, as we now argue.
  If we fix some $\beta \in \mathbb{R}$ and plug $\lambda:= \lfloor \beta n^{2/3} \rfloor /n$ into \eqref{mgf_ER_positive_f}  then we obtain (after some calculation) the formula
\begin{equation}\label{crit_lim}
\lim_{n \to \infty} n^{1/3} \mathbb{E}_{n, p(t,n)} \left[ \exp\left( -\beta u X_n -\frac12 \beta^2 X_n+\frac12 \beta X_n^2 \right)-1 \right] = \beta, \qquad \beta \in \mathbb{R}.
\end{equation}
Now one can use stochastic calculus to show 
\begin{equation}\label{excursion_martingale}
\int_0^\infty \left(\exp\left( -\beta u x - \frac12 \beta^2 x +\frac12 \beta x^2 \right)-1\right) \, \mathrm{d} \rho^u_0(x)=\beta \quad \text{for any} \quad \beta \in \R.
\end{equation}
We conjecture that \eqref{exc_length_limit_comp_size} can be derived from \eqref{crit_lim} and \eqref{excursion_martingale}. 
\end{remark}

 We discuss the origins of \eqref{mgf_ER_component_positive}  in Remark \ref{remark_rigid_extensions}\eqref{original_proof_remark} and an extension of 
 \eqref{mgf_ER_component_positive} to the \emph{stochastic block model} in 
Remark \ref{remark_rigid_extensions}\eqref{extensions_remark}.

\section{Proof of Proposition \ref{prop_ER_mgf_general} }
\label{section_proof_of_combinatorial_prop}

The proof of Proposition \ref{prop_ER_mgf_general} will easily follow from the \emph{change of measure} formula \eqref{eq_change_of_measure_statement}. An idea similar to \eqref{eq_change_of_measure_statement} has already been used in the proof of \cite[Theorem 11.6.1]{alon_spencer_fourth}.

\begin{lemma}\label{lemma_change_of_measure}
For any $M, N \in \N$, $p \in [0,1]$,
 and $k \in \{1,\dots,N\}$
  we have
  \begin{equation}\label{eq_change_of_measure_statement}
  \mathbb{P}_{M,p}[ \, |\mathcal{C}|=k \, ] = \mathbb{P}_{N,p}[ \, |\mathcal{C}|=k \, ] \cdot
  (1-p)^{(M-N)k} \prod_{i=1}^{k-1}   \frac{M-i}{ N -i }.
  \end{equation}
\end{lemma}

\begin{proof} If $k > M$ then both sides of \eqref{eq_change_of_measure_statement}
are  zero.  
Thus w.l.o.g.\ we can assume $k \leq M \wedge N$. Now we observe that 
if we prove
\eqref{eq_change_of_measure_statement} for some $M \leq N$, then we also obtain
\eqref{eq_change_of_measure_statement}  for $M'=N$ and $N'=M$ by rearranging the formula \eqref{eq_change_of_measure_statement}, thus
 we may  assume 
 w.l.o.g.\
that $k \leq M \leq N$.
 In order to prove 
 \eqref{eq_change_of_measure_statement} it is enough to show
\begin{equation}\label{change_of_measure_2}
 \binom{M-1}{k-1}^{-1} \mathbb{P}_{M,p}[ \, |\mathcal{C}|=k \, ] \cdot(1-p)^{k(N-M)}=
 \binom{N-1}{k-1}^{-1} \mathbb{P}_{N,p}[ \, |\mathcal{C}|=k \, ].
\end{equation}
Now if we denote by $V(\mathcal{C})$ the vertex set of $\mathcal{C}$ then
\begin{equation}\label{exchangeability}
\mathbb{P}_{N,p}[ \, |\mathcal{C}|=k \, ]=\binom{N-1}{k-1} 
\mathbb{P}_{N,p}[ \, V(\mathcal{C})=[k]  \, ],
\end{equation}
since $\mathbb{P}_{N,p}$ is invariant under the permutation of vertices and there
are $\binom{N-1}{k-1}$ subsets of $[N]$ with cardinality $k$ that contain the vertex indexed by $1$.
Using \eqref{exchangeability} for $\mathbb{P}_{N,p}$ as well as $\mathbb{P}_{M,p}$, the formula \eqref{change_of_measure_2} reduces to showing
\begin{equation}\label{triv_not_connected}
\mathbb{P}_{M,p}[ \, V(\mathcal{C})=[k]  \, ]\cdot(1-p)^{k(N-M)}= 
\mathbb{P}_{N,p}[ \, V(\mathcal{C})=[k]  \, ].
\end{equation}
Now \eqref{triv_not_connected} holds
since $V(\mathcal{C})=[k]$ in $\mathcal{G}_{N,p}$ if and only if 
$V(\mathcal{C})=[k]$ in $\mathcal{G}_{M,p}$ and there are no edges in $\mathcal{G}_{N,p}$ between
$[k]$ and $[N]\setminus [M]$. This completes the proof of Lemma \ref{lemma_change_of_measure}.
\end{proof}

\begin{proof}[Proof of Proposition \ref{prop_ER_mgf_general}]

For any $n \in \N$, $j \in \Z \cap (-n, +\infty)$ and $p \in [0,1]$ we have
\begin{multline}
\mathbb{E}_{n,p} \left[\, g_{n,p}(j,|\mathcal{C}| ) \right]
\stackrel{ \eqref{g_n_p_k} }{=}
\frac{n+j}{n} \sum_{k=1}^n 
\mathbb{P}_{n,p}[ \, |\mathcal{C}|=k \, ] \cdot
  (1-p)^{jk} \prod_{i=1}^{k-1}   \frac{n+j-i}{ n -i }
  \\ \stackrel{(*)}{=}
\frac{n+j}{n} \sum_{k=1}^n \mathbb{P}_{n+j,p}[ \, |\mathcal{C}|=k \, ]=
\frac{n+j}{n}\left(1-\mathbb{P}_{n+j,p}[ \, |\mathcal{C}|>n \, ] \right),
\end{multline}
where in $(*)$ we used \eqref{eq_change_of_measure_statement} with $n=N$ and $M=n+j$. The proof of 
\eqref{mgf_ER_component_positive} is complete.
\end{proof}

\begin{remark}\label{remark_rigid_extensions}
\begin{enumerate}[(i)]
\item \label{original_proof_remark}  Our original proof of Proposition \ref{prop_ER_mgf_general} used the so-called 
 \emph{rigid representation} of the time evolution of the component size structure of the \ER graph, see \cite[Section 6.1.1, Case 1]{rigid}. In a nutshell,
 if $Y_k=t-X_k$, $k \in [n]$, where
  $X_1,X_2,\dots, X_{n}$ denote independent exponentially distributed random variables
$ X_k \sim \mathrm{EXP}\left(1-\frac{k}{n}\right)$, then 
 $\tau=\min\{ \, k \, : \, Y_1+\dots+Y_k < 0 \, \}$
  has the same distribution as $|\mathcal{C}|$ under $\mathbb{P}_{n,p}, \; p=1-e^{-t/n}$. 
We chose to include an elementary proof instead in order to keep the paper self-contained.
 
\item \label{extensions_remark} It is  possible to extend Proposition \ref{prop_ER_mgf_general}
to the \emph{stochastic block model}, as we now briefly explain. Consider a random graph
in which each vertex has a label, where the set of labels is $\{1,\dots, \ell\}$. Let $\underline{n}=(n_1,\dots,n_\ell)$ and $n=n_1+\dots+n_\ell$. We uniformly choose a labelling of
the vertex set $[n]$ 
 from the set of labellings where the number of vertices with label $j$ is $n_j$ for each $j=1,\dots, \ell$. Given the labels, we add edges independently:
a vertex with label
$i$ and a vertex with label $j$ is connected with probability $p_{i,j}$. 
Let $\underline{\underline{p}}=(p_{i,j})_{i,j=1}^\ell$. 
Denote by $\mathbb{P}_{\underline{n}, \underline{\underline{p}}}$ the law of the resulting random graph 
$\mathcal{G}_{\underline{n}, \underline{\underline{p}}}$ and
$\mathbb{E}_{\underline{n}, \underline{\underline{p}}}$ the corresponding expectation.
This random graph model is often called the \emph{stochastic block model} and it is also
a special case of the \emph{inhomogeneous random graph model} of \cite{inhom}.

Denote by $\mathcal{K}(\underline{n})$ the set of vectors
$\underline{k}=(k_1,\dots,k_\ell)$ for which $0 \leq k_j \leq n_j$ and $k_1+\dots+k_\ell \geq 1$.
Denote by $\mathcal{J}(\underline{n})$ the set of vectors
$\underline{J}=(J_1,\dots,J_\ell)$ for which $-n_j \leq J_j$ and $-n<J_1+\dots+J_\ell$.
Denote by $\mathcal{J}^{\leq 0}(\underline{n})$ the subset of
$\mathcal{J}(\underline{n})$ which consists of vectors
$\underline{J}=(J_1,\dots,J_\ell)$ for which $J_j \leq 0$ for any $j=1,\dots,\ell$.
Let us define
\begin{equation}
g_{\underline{n},\underline{\underline{p}}}(\underline{J},\underline{k}):=
 \prod_{i,j=1}^\ell (1-p_{i,j})^{k_i J_j}  \, \cdot \,
 \prod_{j=1}^\ell \prod_{i=0}^{k_j-1} \frac{n_j+J_j-i}{n_j-i}, \qquad \underline{J} \in \mathcal{J}(\underline{n}), \;\; \underline{k} \in \mathcal{K}(\underline{n}).
\end{equation}
Let $\mathcal{C}$ denote the connected component of the vertex indexed by $1$ in $\mathcal{G}_{\underline{n}, \underline{\underline{p}}}$.
Denote by $|\mathcal{C}|_j$ the number of vertices with label $j$ in $\mathcal{C}$ and let
$\underline{|\mathcal{C}|}=(|\mathcal{C}|_1, \dots, |\mathcal{C}|_\ell)$.
 The generalization of the formula
\eqref{mgf_ER_component_positive} to the stochastic block model is
\begin{equation}\label{stoch_block_magic}
\mathbb{E}_{\underline{n},\underline{\underline{p}}} 
\left[\, g_{\underline{n},\underline{\underline{p}}}(\underline{J},\underline{|\mathcal{C}|} )
  \, \right]=\frac{\sum_{j=1}^\ell (n_j+J_j) }{\sum_{j=1}^\ell n_j}
  \mathbb{P}_{\underline{n}+\underline{J},\underline{\underline{p}}}
  \left[ \,  |\mathcal{C}|_1 \leq n_1, \dots,  |\mathcal{C}|_\ell \leq n_\ell \, \right], \;
  \underline{J} \in \mathcal{J}(\underline{n}).
\end{equation}
 In order to prove \eqref{stoch_block_magic}, one needs the
 following analogue of \eqref{eq_change_of_measure_statement}, valid for $\underline{k} \in \mathcal{K}(\underline{N})$:
 \begin{equation}\label{eq_change_of_measure_stoch_block}
  \mathbb{P}_{\underline{M},\underline{\underline{p}} }\left[ \, \underline{|\mathcal{C}|}=
  \underline{k} \, \right] = 
  \mathbb{P}_{\underline{N},\underline{\underline{p}}}\left[ \, \underline{|\mathcal{C}|}=\underline{k} \, \right]  
  \prod_{i,j=1}^\ell (1-p_{i,j})^{(M_j-N_j)k_j} \,
  \frac{\sum_{j=1}^\ell N_j  }{ \sum_{j=1}^\ell M_j }
   \prod_{j=1}^\ell \prod_{i=0}^{k_j-1} \frac{M_j-i}{N_j-i}
    \end{equation}
Note that if $\underline{J} \in \mathcal{J}^{\leq 0}(\underline{n})$ then the r.h.s.\ of \eqref{stoch_block_magic} is simply $\frac{\sum_{j=1}^\ell (n_j+J_j) }{\sum_{j=1}^\ell n_j}$. Also
note that the analogue of the property stated in Remark \ref{remark_unique_characterization}\eqref{rem_unique} holds: 
the system of equations \eqref{stoch_block_magic} indexed by $\underline{J} \in \mathcal{J}^{\leq 0}(\underline{n})$ uniquely characterizes the distribution of 
$\underline{|\mathcal{C}|}$ under $\mathbb{P}_{\underline{n},\underline{\underline{p}}}$.
 \end{enumerate}
\end{remark}

\section{Proof of Theorem \ref{thm_subcrit}}
\label{section:subcrit}

The basic idea is to treat $\mathbb{E}_{n,p} \left[ g_{n,p}(j,|\mathcal{C}| ) \right]$ as 
the generating function of $|\mathcal{C}|$, c.f. Remark \ref{remark_genereting_fn_borel}. Thus if we want to obtain information about the first and second moments
of $|\mathcal{C}|$, we have to ``differentiate'' with respect to the variable $j$ twice.
Since  $j$ can only take integer values, we have to consider the first order discrete differences
$g_{n,p}(j,|\mathcal{C}|) - g_{n,p}(0,|\mathcal{C}|)$ for $j=-1$ and $j=-2$ in the proof of Lemmas \ref{lemma_third_order_formula_minus_one} and \ref{lemma_third_order_formula_minus_two},
and the second order discrete difference (i.e., the difference of the first order differences)
in the proof of Lemma \ref{lemma_second_derivative}.

The statement of Lemma \ref{lemma_domination_moment_bounds} is equivalent to
\cite[Lemma 3.2]{janson_luczak} (which is proved using differential equations), moreover it also classically follows from the fact that
$|\mathcal{C}|$ is stochastically dominated by a subcritical branching process if $t<1$.
Despite of this, we chose to include a proof of Lemma \ref{lemma_domination_moment_bounds} which only uses
Proposition \ref{prop_ER_mgf_general} in order to keep the paper self-contained.

Recall our convention $p=1-e^{-t/n}$ from \eqref{p_t_n_def}.
\begin{lemma}\label{lemma_domination_moment_bounds}
If $t\in (0,1) $ then
\begin{equation}\label{eq_bp_moment_bounds}
\mathbb{E}_{n,p}\left( |\mathcal{C}|^i \right) =\mathcal{O}\left( \frac{1}{(1-t)^{2i-1}} \right), \qquad i\in \mathbb{N}.
\end{equation}
\end{lemma}

\begin{proof} W.l.o.g. we assume $t \in [\frac12, 1)$ and $\frac{2}{1-t} \leq n$. For any $j \geq 0$ we have
\begin{equation}\label{subcrit_gen_fn_bound}
\mathbb{E}_{n,p}\left[  \left( e^{-tj/n} \left( 1+ \frac{j}{n} \right) \right)^{|\mathcal{C}|} \right] 
\stackrel{ \eqref{f_n_lambda_k} }{\leq}  \mathbb{E}_{n,p}\left[\, f_{n,t}\left(\frac{j}{n},|\mathcal{C}|\right) \,\right]  
\stackrel{ \eqref{mgf_ER_positive_f} }{ \leq } 1+ \frac{j}{n}.
\end{equation}
Note that if we let $\widetilde{\lambda}:=\frac{1}{t}-1$ then we have 
  \begin{equation}\label{widetilde_lambda_def_eq}
  \max_{\lambda}e^{-\lambda t}(1+\lambda)=   e^{-\widetilde{\lambda}t}(1+\widetilde{\lambda})=\frac{1}{t}e^{t-1}\stackrel{(*)}{>}e^{\frac{1}{2}(1-t)^2} \quad t\in (0,1),
  \end{equation}
  where $(*)$ follows from $-\ln(t)+(t-1)-\frac{1}{2}(t-1)^2=\int_t^1 \int_s^1 \left( \frac{1}{u^2}-1\right)
  \, \mathrm{d}u \, \mathrm{d}s>0$.
  
Next we show that if we choose $j^* :=\lfloor n \widetilde{\lambda} \rfloor= \lfloor n \cdot \left( \frac{1}{t} -1 \right) \rfloor$  then we have
\begin{equation}\label{loc_max_subcrit_bound}
e^{-tj^*/n} \left( 1+ \frac{j^*}{n} \right) \geq  e^{\frac14 (1-t)^2}.
\end{equation}
Indeed, if we let $f(x):=-tx+\ln(1+x)$, then we have $f'(\widetilde{\lambda})=0$ and thus
\begin{equation}\label{calc_perturb}
f\left(\widetilde{\lambda}\right)-f\left(\frac{j^*}{n}\right)=
\int_{j^*/n}^{\widetilde{\lambda}}\int_{x}^{\widetilde{\lambda}} -f''(y) \, \mathrm{d}y \, \mathrm{d}x=
\int_{j^*/n}^{\widetilde{\lambda}} \frac{y-j^*/n}{(1+y)^2}\mathrm{d}y \stackrel{(**)}{\leq} \frac{1}{n^2}
\stackrel{(***)}{\leq} \frac14 (1-t)^2,
\end{equation}
where $(**)$ follows from $0 \leq j^*$ and $0 \leq \widetilde{\lambda} - j^*/n \leq 1/n$, and
$(***)$ follows from $\frac{2}{1-t} \leq n$.
Now \eqref{loc_max_subcrit_bound} follows from \eqref{widetilde_lambda_def_eq} and 
\eqref{calc_perturb}.  We are now ready to prove \eqref{eq_bp_moment_bounds}:
\begin{multline}
1+\frac{1}{i!} \frac{1}{4^i} (1-t)^{2i} \mathbb{E}_{n,p}\left( |\mathcal{C}|^i \right)
\leq
 \mathbb{E}_{n,p} \left[ \sum_{\ell=0}^{\infty} \frac{ \left(\frac{1}{4}(1-t)^2 |\mathcal{C}| \right)^\ell }{\ell!} \right] \\=
  \mathbb{E}_{n,p} \left[ e^{\frac14 (1-t)^2 |\mathcal{C}|} \right] 
  \stackrel{ \eqref{loc_max_subcrit_bound} }{\leq}
 \mathbb{E}_{n,p}\left[  \left( e^{-tj^*/n} \left( 1+ \frac{j^*}{n} \right) \right)^{|\mathcal{C}|} \right]
  \stackrel{ \eqref{subcrit_gen_fn_bound}   }{\leq} 
   1+ \frac{j^*}{n} \leq
  \frac{1}{t} \qquad i \in \N,
\end{multline}
from which \eqref{eq_bp_moment_bounds} follows if $t \in [\frac12, 1)$. 
\end{proof}

\begin{lemma}\label{lemma_third_order_formula_minus_one}
 For any $t \in [0,1)$ we have
\begin{equation}\label{j_minus_1_formula_third_o}
 1=(1-t)\mathbb{E}_{n,p}(|\mathcal{C}|)+
\left(t-\frac{t^2}{2} \right)\frac{\mathbb{E}_{n,p}(|\mathcal{C}|^2)}{n} +
\left(\frac{t^2}{2}- \frac{t^3}{6} \right)\frac{\mathbb{E}_{n,p}(|\mathcal{C}|^3)}{n^2} 
+\mathcal{O}\left( \frac{1}{(1-t)^7 n^3} \right).
\end{equation}
\end{lemma}
Before we prove Lemma \ref{lemma_third_order_formula_minus_one}, let us state an immediate corollary.
\begin{corollary} Applying \eqref{eq_bp_moment_bounds} to $\mathbb{E}_{n,p}(|\mathcal{C}|^3)$ in
\eqref{j_minus_1_formula_third_o} we obtain
\begin{equation}\label{second_order_expansion_of_first_moment}
\mathbb{E}_{n,p}(|\mathcal{C}|)=\frac{1}{1-t}+
\frac{\frac{t^2}{2}-t }{1-t}\frac{\mathbb{E}_{n,p}(|\mathcal{C}|^2)}{n} +
\mathcal{O}\left( \frac{1}{(1-t)^6n^2} \right), \quad t \in [0, 1-n^{-1/3}].
\end{equation}
Applying \eqref{eq_bp_moment_bounds} to $\mathbb{E}_{n,p}(|\mathcal{C}|^2)$ in
\eqref{second_order_expansion_of_first_moment} we obtain
\begin{equation}\label{first_order_expansion_of_first_moment}
\mathbb{E}_{n,p}(|\mathcal{C}|)=\frac{1}{1-t}+
\mathcal{O}\left( \frac{1}{(1-t)^4 n}\right), \quad t \in [0, 1-n^{-1/3}].
\end{equation}

\end{corollary}

 \begin{proof}[Proof of Lemma \ref{lemma_third_order_formula_minus_one}]
 Let $k \in [n]$.
We begin by observing that \eqref{g_n_p_k} is a telescopic product if $j=-1$ and then we apply
Taylor expansion:
 \begin{multline}\label{third_order_expansion_of_g_minus_one}
g_{n,p}(-1,k) \stackrel{ \eqref{g_n_p_k}, \eqref{p_t_n_def} }{=}  e^{tk/n}\left( 1 -\frac{k}{n} \right)=
\left(\sum_{i=0}^3 \frac{1}{i!}  \frac{t^i k^i }{n^i}  + \mathcal{O}\left( \frac{k^4}{n^4} \right)  \right)\left( 1 -\frac{k}{n} \right)\\
=1+(t-1)\frac{k}{n}+\left(\frac{t^2}{2}-t \right)\frac{k^2}{n^2}+ 
\left(\frac{t^3}{6}-\frac{t^2}{2} \right)\frac{k^3}{n^3}
 +\mathcal{O}\left( \frac{k^4}{n^4} \right).
\end{multline}
Combining \eqref{third_order_expansion_of_g_minus_one} with Proposition \ref{prop_ER_mgf_general} we obtain
\begin{multline}\label{g_minus_one_taylor_expect} 
1-\frac{1}{n}=
1+(t-1)\frac{\mathbb{E}_{n,p}(|\mathcal{C}|)}{n}
\\+
\left(\frac{t^2}{2}-t \right)\frac{\mathbb{E}_{n,p}(|\mathcal{C}|^2)}{n^2}+
\left(\frac{t^3}{6}-\frac{t^2}{2} \right)\frac{\mathbb{E}_{n,p}(|\mathcal{C}|^3)}{n^3} 
+\mathcal{O}\left( \frac{\mathbb{E}_{n,p}(|\mathcal{C}|^4)}{n^4} \right).
\end{multline}
Subtracting one from both sides of \eqref{g_minus_one_taylor_expect}, multiplying the result by $-n$
 and applying \eqref{eq_bp_moment_bounds} to $\mathbb{E}_{n,p}(|\mathcal{C}|^4)$, we obtain \eqref{j_minus_1_formula_third_o}.
\end{proof}

\begin{lemma}\label{lemma_third_order_formula_minus_two}
 For any $t \in [0,1-n^{-1/3} ]$ we have
\begin{multline}\label{j_minus_2_formula_trird_order_polynomial}
-2=(2t-2)\mathbb{E}_{n,p}[|\mathcal{C}|]+(1-4t+2t^2)\frac{\mathbb{E}_{n,p}[|\mathcal{C}|^2]}{n} -\frac{\mathbb{E}_{n,p}[|\mathcal{C}|]}{n}  \\ +
(2t-4t^2+\frac{4}{3}t^3)\frac{\mathbb{E}_{n,p}[|\mathcal{C}|^3]}{n^2}+
\mathcal{O}\left( \frac{1}{(1-t)^7 n^3} \right).
\end{multline}

\end{lemma}

\begin{proof}
Let $k \in [n]$.
 We begin with a calculation similar to \eqref{third_order_expansion_of_g_minus_one}:
\begin{multline}\label{third_order_expansion_of_g_minus_two}
g_{n,p}(-2,k)
\stackrel{ \eqref{g_n_p_k}, \eqref{p_t_n_def} }{=}
e^{2tk/n}\left(1-\frac{k}{n} \right)\left(1-\frac{k}{n-1}\right)\\=
\left( \sum_{i=0}^3 \frac{1}{i!}  \frac{2^i t^i k^i}{n^i}  +\mathcal{O}\left( \frac{k^4}{n^4}\right)
 \right)
 \left(1-\frac{k}{n} \right)\left(1-\frac{k}{n}
 \left( \sum_{j=0}^2 \frac{1}{n^j} +
\mathcal{O}\left( \frac{1}{n^3} \right)
 \right) \right)\\=
1+(2t-2)\frac{k}{n}+(1-4t+2t^2)\frac{k^2}{n^2} -\frac{k}{n^2}  +
(2t-4t^2+\frac{4}{3}t^3)\frac{k^3}{n^3}+(1-2t)\frac{k^2}{n^3}-\frac{k}{n^3}+
\mathcal{O}\left( \frac{k^4}{n^4} \right).
\end{multline}
From \eqref{third_order_expansion_of_g_minus_two} and Proposition \ref{prop_ER_mgf_general} we obtain
\begin{multline}\label{g_minus_two_taylor_expect}
1-\frac{2}{n}=1+(2t-2)\frac{\mathbb{E}_{n,p}[|\mathcal{C}|]}{n}+(1-4t+2t^2)\frac{\mathbb{E}_{n,p}[|\mathcal{C}|^2]}{n^2} -\frac{\mathbb{E}_{n,p}[|\mathcal{C}|]}{n^2}  \\ +
(2t-4t^2+\frac{4}{3}t^3)\frac{\mathbb{E}_{n,p}[|\mathcal{C}|^3]}{n^3}+
(1-2t)\frac{\mathbb{E}_{n,p}[|\mathcal{C}|^2]}{n^3}-\frac{\mathbb{E}_{n,p}[|\mathcal{C}|]}{n^3}+
\mathcal{O}\left( \frac{\mathbb{E}_{n,p}[|\mathcal{C}|^4]}{n^4} \right).
\end{multline}
Subtracting one from both sides of \eqref{g_minus_two_taylor_expect}, multiplying the result by $n$
 and applying \eqref{eq_bp_moment_bounds} to the last three terms of \eqref{g_minus_two_taylor_expect}, we obtain \eqref{j_minus_2_formula_trird_order_polynomial}.

\end{proof}

\begin{lemma}\label{lemma_second_derivative}
For any $t \in [0, 1-n^{-1/3}]$ we have
\begin{equation}\label{first_second_moment_compare_using_third_order}
\mathbb{E}_{n,p}(|\mathcal{C}|^2)
=
\frac{\mathbb{E}_{n,p}(|\mathcal{C}|)}{(1-t)^2}+
 \mathcal{O} \left(\frac{1}{(1-t)^6 n} \right).
\end{equation}
\end{lemma}
\begin{proof}
Adding  \eqref{j_minus_2_formula_trird_order_polynomial} to twice \eqref{j_minus_1_formula_third_o} we obtain
\begin{equation}\label{adding_minus_one_and_minus_two}
0=(1-2t+t^2) \frac{\mathbb{E}_{n,p}(|\mathcal{C}|^2)}{n}-\frac{\mathbb{E}_{n,p}(|\mathcal{C}|)}{n}+
(t^3-3t^2+2t)\frac{\mathbb{E}_{n,p}(|\mathcal{C}|^3)}{n^2}+
\mathcal{O}\left( \frac{1}{(1-t)^7 n^3} \right).
\end{equation}
Rearranging \eqref{adding_minus_one_and_minus_two} and multiplying by $n$ we obtain
\begin{equation}\label{first_second_moment_with_third_cancel}
\mathbb{E}_{n,p}(|\mathcal{C}|)=(1-t)^2\mathbb{E}_{n,p}(|\mathcal{C}|^2)+
t(t-1)(t-2)\frac{\mathbb{E}_{n,p}(|\mathcal{C}|^3)}{n}+\mathcal{O}\left( \frac{1}{(1-t)^7 n^2} \right).
\end{equation}
Dividing both sides of \eqref{first_second_moment_with_third_cancel} by $(1-t)^2$
we use \eqref{eq_bp_moment_bounds} to obtain \eqref{first_second_moment_compare_using_third_order}.
\end{proof}

\begin{proof}[Proof of Theorem \ref{thm_subcrit}]

From \eqref{first_order_expansion_of_first_moment} and \eqref{first_second_moment_compare_using_third_order}  
we obtain \eqref{janson_second_moment_error}.

Plugging \eqref{janson_second_moment_error}  into \eqref{second_order_expansion_of_first_moment} we obtain \eqref{tim_error_susceptibility}.

\end{proof}

\section{Proof of Theorem \ref{thm_CLT_giant} }
\label{section:clt}

We will deduce Theorem \ref{thm_CLT_giant}  (i.e., the CLT for $|\mathcal{C}_{max}|$) from
Lemma \ref{lemma_fixed_component_clt} (i.e., the CLT for $|\mathcal{C}|$) using the idea
of \cite[Lemma 2.1]{remco_local_limit}. 
We deduce Lemma \ref{lemma_fixed_component_clt} from Lemmas \ref{lemma_C_falls_in_first_interval_with_one_minus_theta} and \ref{lemma_clt_mgf_widetilde_I_n}
using that convergence of moment generating functions implies
weak convergence of probability distributions. We prove
Lemmas \ref{lemma_C_falls_in_first_interval_with_one_minus_theta}
and \ref{lemma_clt_mgf_widetilde_I_n} 
 by viewing \eqref{mgf_ER_positive_f} as a
moment generating function identity. 
The crux of the proof of Lemma \ref{lemma_C_falls_in_first_interval_with_one_minus_theta} is \eqref{f_n_minus_theta} and the
crux of the proof of Lemma \ref{lemma_clt_mgf_widetilde_I_n} is \eqref{giant_square_root_window}.
 
 \medskip

Throughout this section we fix $t>1$. Recall the notion of  $\varphi: [0,1) \to \R$
from \eqref{varphi_def} and  $\theta=\theta(t)\in (0,1)$ from \eqref{theta_def}.
Recall the notion of $p=p(t,n)=1-e^{-t/n}$ from \eqref{p_t_n_def}. 

We will often use the shorthand $\mathbb{P}$ for $\mathbb{P}_{n,p(t,n)}$ and $\mathbb{E}$ for 
$\mathbb{E}_{n,p(t,n)}$. 

If $X$ is a random variable and $A$ is an event,  we will denote 
$\mathbb{E}( X ; A):=\mathbb{E}( X \mathds{1}_{A})$.

\begin{lemma} \label{lemma_fixed_component_clt}
Let us define $\sigma$ as in \eqref{clt_giant_formula}.
For any $x \in \R$  we have
\begin{equation}\label{fixed_component_clt_eq}
 \lim_{n \to \infty}
 \mathbb{P}_{n,p(t,n)}\left[ \, \frac{ |\mathcal{C}|-\theta n }{\sigma \sqrt{n}} 
 \leq x \,
  \right]= (1-\theta)+ \theta \Phi(x).
\end{equation}
\end{lemma}
Before we prove Lemma \ref{lemma_fixed_component_clt}, we use it to prove Theorem \ref{thm_CLT_giant}.

\begin{proof}[Proof of Theorem \ref{thm_CLT_giant}] Denote by $|\mathcal{C}_1|, |\mathcal{C}_2|, \dots$
the non-increasing rearrangement of the sequence of component sizes of the graph $\mathcal{G}_{n,p}$. Thus
$|\mathcal{C}_1|=|\mathcal{C}_{max}|$ and $|\mathcal{C}_2|$ is the size of the second largest component.
Note that $|\mathcal{C}_1|= |\mathcal{C}_2|$ is possible, but we will show that 
$|\mathcal{C}_2|< |\mathcal{C}_1|$ with high probability.

For any $a \in \mathbb{R}$ let us denote 
\begin{equation}\label{kna}
k_{n,a}= \lfloor \theta n + a \cdot \sigma \sqrt{n} \rfloor.
\end{equation}
  We will show that for $a \leq  b  \in \mathbb{R}$ we have
\begin{equation}\label{clt_with_second_largest}
\lim_{n \to \infty}
 \mathbb{P}_{n, p(t,n)}\left[ \, |\mathcal{C}_1| \in [k_{n,a},k_{n,b}], \, |\mathcal{C}_2|< k_{n,a} \,
  \right] 
  = \frac{1}{\theta} \lim_{n \to \infty}
 \mathbb{P}_{n, p(t,n)}\left[ \, |\mathcal{C}| \in [k_{n,a},k_{n,b}] \,\right].
\end{equation}
Now by Lemma \ref{lemma_fixed_component_clt} the right-hand side of \eqref{clt_with_second_largest} is 
$\Phi\left(b\right)-\Phi\left(a \right)$. This equation readily implies 
$\liminf_{n \to \infty}
 \mathbb{P}_{n, p(t,n)}\left[ \, |\mathcal{C}_2|< k_{n,a} \,
  \right] \geq \Phi\left( b \right)-\Phi\left( a \right)$
  for any $a \leq  b \in \mathbb{R}$, which in turn implies 
  $\lim_{n \to \infty}
 \mathbb{P}_{n, p(t,n)}\left[ \, |\mathcal{C}_2|< k_{n,a} \,
  \right]=1$ for any $a \in \mathbb{R}$. Combining this with Lemma \ref{lemma_fixed_component_clt} and
  \eqref{clt_with_second_largest} we obtain that  Theorem \ref{thm_CLT_giant} indeed holds.
  
  In order to prove \eqref{clt_with_second_largest} we observe that if $k \in [k_{n,a},k_{n,b}]$, then
  \begin{equation}\label{remco_trick}
  \mathbb{P}_{n,p}\left[ \, |\mathcal{C}_1| =k , \, |\mathcal{C}_2|< k_{n,a} \,
  \right]= \mathbb{P}_{n-k,p}\left[ \,  |\mathcal{C}_1| < k_{n,a} \, \right]
  \frac{n}{k} \mathbb{P}_{n,p}\left[\, |\mathcal{C}|=k  \, \right].
  \end{equation}
  Equation \eqref{remco_trick} is essentially a special case of \cite[Lemma 2.1]{remco_local_limit}, but
  we include the proof of \eqref{remco_trick} here for completeness: if $v \in [n]$ and $\mathcal{C}(v)$ denotes the connected component of $v$ in $\mathcal{G}_{n,p}$ and $|\mathcal{C}^*(v)|$ denotes the size of the largest 
  connected component of $\mathcal{G}_{n,p} \setminus \mathcal{C}(v)$ then
\begin{multline*}
\mathbb{P}_{n,p}\left[ \, |\mathcal{C}_1| =k , \, |\mathcal{C}_2|< k_{n,a} \,
  \right] =  \frac{1}{k} \sum_{v=1}^n 
  \mathbb{P}_{n,p}\left[ \, |\mathcal{C}_1| =k , \, v \in \mathcal{C}_1, \, |\mathcal{C}_2|< k_{n,a} \, 
  \right] \\=
  \frac{1}{k} \sum_{v=1}^n 
  \mathbb{P}_{n,p}\left[ \, |\mathcal{C}(v)|=k , \, |\mathcal{C}^*(v)|< k_{n,a} \,
  \right]=\frac{n}{k} \mathbb{P}_{n,p}\left[ \, |\mathcal{C}(1)|=k , \, |\mathcal{C}^*(1)|< k_{n,a} \,
  \right]\\=
 \frac{n}{k}  \mathbb{P}_{n,p}\left[\, |\mathcal{C}|=k  \, \right]  
  \mathbb{P}_{n-k,p}\left[ \,  |\mathcal{C}_1| < k_{n,a} \, \right].
\end{multline*}  
This proves  \eqref{remco_trick}. Next we show that 
\begin{equation}\label{second_max_small}
\lim_{n \to \infty} \min_{k \in [k_{n,a},k_{n,b}]} \mathbb{P}_{n-k,p(t,n)}\left[ \,  |\mathcal{C}_1| < k_{n,a} \, \right]=1.
\end{equation}
Let us denote $\widetilde{n}=n-k_{n,a}$.
 For any $k \in [k_{n,a},k_{n,b}]$ we have
\begin{equation}\label{second_max_bounds}
\mathbb{P}_{n-k,p}\left[ \,  |\mathcal{C}_1| \geq k_{n,a} \, \right] \leq
\frac{n-k}{k_{n,a}} \mathbb{P}_{n-k,p}\left[ \,  |\mathcal{C}| \geq k_{n,a} \, \right] \leq 
\frac{n }{k_{n,a}}
\mathbb{P}_{\widetilde{n},p}\left[ \,  |\mathcal{C}| \geq k_{n,a} \, \right] \leq
\frac{n \mathbb{E}_{\widetilde{n},p}\left[ \,  |\mathcal{C}| \, \right]}{(k_{n,a})^2}.
\end{equation}

Now we observe that $\mathcal{G}_{\widetilde{n},p(t,n)}$ is a \emph{subcritical} \ER graph, since
\[\lim_{n \to \infty} \widetilde{n} \cdot p(t,n) \stackrel{ \eqref{p_t_n_def}, \eqref{kna} }{=}
\lim_{n \to \infty} \left( n- \lfloor \theta n + a \cdot \sigma \sqrt{n} \rfloor \right) \cdot (1-e^{-t/n})
=
(1-\theta)t \stackrel{ \eqref{varphi_prime_theta} }{<}1.\]
Note that $\mathbb{E}_{\widetilde{n},p} \left[ |\mathcal{C}| \right]$ remains bounded as $n \to \infty$
 by \eqref{eq_bp_moment_bounds}, hence
 \eqref{second_max_small}  follows from \eqref{second_max_bounds}.
 
We are now ready to prove \eqref{clt_with_second_largest}:
\begin{multline}
\lim_{n \to \infty}
 \mathbb{P}_{n, p(t,n)}\left[ \, |\mathcal{C}_1| \in [k_{n,a},k_{n,b}], \, |\mathcal{C}_2|< k_{n,a} \,
  \right] \\ \stackrel{ \eqref{remco_trick} }{=} \lim_{n \to \infty} \sum_{k=k_{n,a}}^{k_{n,b}} \mathbb{P}_{n-k,p(t,n)}\left[ \,  |\mathcal{C}_1| < k_{n,a} \, \right]
  \frac{n}{k} \mathbb{P}_{n,p(t,n)}\left[\, |\mathcal{C}|=k  \, \right] \\
  \stackrel{ \eqref{kna}, \eqref{second_max_small} }{=}
  \frac{1}{\theta} \lim_{n \to \infty}
 \mathbb{P}_{n, p(t,n)}\left[ \, |\mathcal{C}| \in [k_{n,a},k_{n,b}] \,\right].
\end{multline}

This completes the proof of Theorem \ref{thm_CLT_giant} given Lemma \ref{lemma_fixed_component_clt}.
\end{proof}

%%%%%%%%%%%%%%%%%%%%%%%%%%%%%%%%%%%%%%%%%%

We will deduce Lemma \ref{lemma_fixed_component_clt} from Lemmas \ref{lemma_C_falls_in_first_interval_with_one_minus_theta} and \ref{lemma_clt_mgf_widetilde_I_n} below.

Let us subdivide the interval $[n]$ into five disjoint sub-intervals:
\begin{align} \label{I_J_K_intervals_def}
 &I_n = [1, n^{1/4}  ), \quad J_n=[n^{1/4},n^{3/4}), \quad
 K_{n} =[ n^{3/4}  , \theta n - n^{5/8} ),
 \\
\label{widetilde_I_K_intervals_def}
 &\widetilde{I}_{n} = [ \theta n - n^{5/8}, \theta n + n^{5/8} ),
 \quad
 \widetilde{K}_{n} = [ \theta n + n^{5/8}, n ].
\end{align}

Note that the choice of the exponents $\frac14$, $\frac34$ and $\frac58$ above is somewhat arbitrary.
Also note that $I_n$ and $\widetilde{I}_n$ are the important intervals, while $J_n$, $K_n$ and $\widetilde{K}_{n}$ are insignificant, i.e., we will see that $|\mathcal{C}| \in I_n \cup \widetilde{I}_n$
with high probability. The only reason behind the distinction between $J_n$ and $K_n$ is that
we will use different methods to show that $J_n$ and $K_n$ are insignificant.

\begin{lemma}\label{lemma_C_falls_in_first_interval_with_one_minus_theta}
 We have 
\begin{equation}\label{C_falls_in_first_interval_with_one_minus_theta}
\lim_{n \to \infty} \mathbb{P}_{n, p(t,n)} \left( |\mathcal{C}| \in I_n \right)=1-\theta.
\end{equation}
\end{lemma}

\begin{lemma}\label{lemma_clt_mgf_widetilde_I_n}
For any $\alpha \in \R$ we have
\begin{equation}\label{clt_mgf_widetilde_I_n}
\lim_{n \to \infty} \left( 
\mathbb{P}_{n, p(t,n)} \left( |\mathcal{C}| \in I_n \right)+  
\mathbb{E}_{n, p(t,n)} \left( \exp\left(\alpha \varphi'(\theta) \frac{|\mathcal{C}|-\theta n }{\sqrt{n}}
- \frac{\alpha^2}{2} \frac{\theta}{1-\theta} \right)
 \, ; \;
 |\mathcal{C}| \in \widetilde{I}_n   \right) \right)= 
 1.
\end{equation}
\end{lemma}

Before we prove Lemmas \ref{lemma_C_falls_in_first_interval_with_one_minus_theta} and
\ref{lemma_clt_mgf_widetilde_I_n}, let us deduce Lemma \ref{lemma_fixed_component_clt} from them.

\begin{proof}[Proof of Lemma \ref{lemma_fixed_component_clt}]
First note that $\lim_{n \to \infty}\mathbb{P}_{n, p(t,n)} \left( |\mathcal{C}| \in I_n \cup \widetilde{I}_n \right)=1$ follows from the $\alpha=0$ case of \eqref{clt_mgf_widetilde_I_n}. Combining this with
\eqref{C_falls_in_first_interval_with_one_minus_theta} we obtain
\begin{equation}
\label{C_falls_in_widetilde_I_with__theta}
\lim_{n \to \infty} \mathbb{P}_{n, p(t,n)} \left( |\mathcal{C}| \in \widetilde{I}_n \right)=\theta.
\end{equation}

Denote by
$\mu_n$ the conditional distribution of $\frac{|\mathcal{C}|-\theta n}{\sqrt{n}}$ given 
$|\mathcal{C}| \in \widetilde{I}_n$. We have
\begin{equation}\label{mu_n_mgf_converges}
\lim_{n \to \infty}
 \int \exp\left(\alpha \varphi'(\theta) x \right) \mathrm{d} \mu_n(x)
 \stackrel{\eqref{C_falls_in_first_interval_with_one_minus_theta}, \eqref{clt_mgf_widetilde_I_n},
 \eqref{C_falls_in_widetilde_I_with__theta} }{=} 
\exp\left( \frac{\alpha^2}{2} \frac{\theta}{1-\theta} \right), \qquad \alpha \in \R.
\end{equation}
The r.h.s.\ of \eqref{mu_n_mgf_converges} is
the  moment generating
function of $\mathcal{N}\left(0, \frac{\theta}{1-\theta}\right)$,
thus it classically follows from \eqref{mu_n_mgf_converges}  that $\mu_n$ weakly converges  to $\mathcal{N}\left(0, \sigma^2\right)$ as $n \to \infty$,
where $\sigma$ appears in \eqref{clt_giant_formula}.
Together with \eqref{C_falls_in_first_interval_with_one_minus_theta} and \eqref{C_falls_in_widetilde_I_with__theta} this
 implies  Lemma \ref{lemma_fixed_component_clt}, given 
 Lemmas \ref{lemma_C_falls_in_first_interval_with_one_minus_theta} and
\ref{lemma_clt_mgf_widetilde_I_n}.
\end{proof}

%%%%%%%%%%%%%%%%%%%%%%%%%%%%%%%%%%%%%%%%%%%%%%%%%%%%%%%%%%%%%%%%%%%%%%%

We will prove Lemma \ref{lemma_C_falls_in_first_interval_with_one_minus_theta} in Section \ref{subsection_proof_lemma_C_falls_in_first_interval_with_one_minus_theta}
and Lemma \ref{lemma_clt_mgf_widetilde_I_n} in Section \ref{subsection_proof_of_lemma_clt_mgf_widetilde_I_n}.
The proofs will make excessive use of \eqref{mgf_ER_positive_f}. Let us now introduce some 
notation that will be used throughout.

For any $\lambda \in (-1, +\infty)$ and any $n \in \N$ let us define
\begin{equation}\label{star_n}
\lambda^*_n=\frac{1}{n} \lfloor n \lambda \rfloor. 
\end{equation}
Now $\lambda^*_n  \in \frac{\Z}{n} \cap (-1,+\infty)$, which is required if we want to use
\eqref{mgf_ER_positive_f}.

Having fixed $t>1$, we note that $\lambda^*_n$ approximates $\lambda$ well, i.e., we have
\begin{equation}\label{f_n_lambda_perturb_compare}
f_{n,t}(\lambda^*_n,k) 
\stackrel{ \eqref{f_n_lambda_k} }{=} 
f_{n,t}(\lambda,k)\exp \left(\mathcal{O}\left( \frac{k}{n} \right) \right), 
\quad e^{-t}-1 \leq \lambda \leq 1, \quad 1 \leq k \leq \frac{e^{-t}}{2}n.
\end{equation}

We will often implicitly use that for any $\lambda >-1$ we have 
\begin{equation}
f_{n,t}(\lambda^*_n,k)=0 \quad \text{if} \quad n +  \lfloor n \lambda \rfloor  < k \leq n
\quad \text{and} \quad
f_{n,t}(\lambda^*_n,k) \geq 0 \quad \text{if} \quad k \in [n].
\end{equation}

Having fixed $t>1$, we note that if we let 
\begin{equation}\label{widetilde_lambda}
 \widetilde{\lambda}:=\frac{1}{t}-1 \quad \text{then we have} 
  \quad \widetilde{x}:=\max_{\lambda}e^{-\lambda t}(1+\lambda)= e^{-\widetilde{\lambda}t}(1+\widetilde{\lambda})=\frac{1}{t}e^{t-1}\stackrel{(*)}{>}1,
\end{equation}
where $(*)$  follows from the inequality $e^x>1+x$ applied to $x=t-1$.

 In Sections \ref{subsection_proof_lemma_C_falls_in_first_interval_with_one_minus_theta} and \ref{subsection_proof_of_lemma_clt_mgf_widetilde_I_n} we will dominate $f_{n,t}(\lambda,k)$
 by $f_{n,t}(\widetilde{\lambda}^*_n,k)$ for $k \in J_n$ (defined in \eqref{I_J_K_intervals_def})  in order to
show that ``nothing interesting happens'' in the interval $J_n$.

We will write $f(n)=\Omega\left( g(n) \right)$ if there
exists a constant $c>0$ (that may depend on $t$) such that $f(n) \geq c g(n)$ for any $n \in \N$.

\subsection{ Proof of Lemma \ref{lemma_C_falls_in_first_interval_with_one_minus_theta}}
\label{subsection_proof_lemma_C_falls_in_first_interval_with_one_minus_theta}

Before we outline the strategy of the proof of Lemma \ref{lemma_C_falls_in_first_interval_with_one_minus_theta} in the paragraph below \eqref{split_components_theta},
we need to introduce some notation. Let us abbreviate 
\[X= f_{n,t}( -\theta  ,|\mathcal{C}| ) \qquad
 \text{and} \qquad X^* \stackrel{ \eqref{star_n} }{=} f_{n,t}( (- \theta )^*_n  ,|\mathcal{C}| ).
 \]
 Recalling the definition of the intervals $I_n$ and $J_n$ from \eqref{I_J_K_intervals_def}, we have
\begin{equation}\label{split_components_theta}
 1+(- \theta )^*_n 
\stackrel{ \eqref{mgf_ER_positive_f} }{=}
  \mathbb{E} \left[\, X^*
   ; \, |\mathcal{C}| \in I_n \right]+
  \mathbb{E} \left[\, X^*
   ; \,   |\mathcal{C}| \in J_n  \right] +
 \mathbb{E} \left[\,  X^*
   ; \,  n^{3/4} \leq |\mathcal{C}|  \right].
\end{equation}
We will estimate the  three terms on the r.h.s.\ of \eqref{split_components_theta}.
We will show that the first term  approximates $\mathbb{P} \left( |\mathcal{C}| \in I_n \right)$
as $n \to \infty$, while the second and third terms vanish as $n \to \infty$.

\begin{proof}[Proof of Lemma \ref{lemma_C_falls_in_first_interval_with_one_minus_theta}]
Before we start estimating the three terms of \eqref{split_components_theta}, we observe
\begin{equation}\label{f_n_minus_theta}
 f_{n,t}(-\theta,k)\stackrel{  \eqref{f_n_lambda_k}, \eqref{theta_def}  }{=}
\prod_{i=0}^{k-1}
 \left( 1- \frac{\theta}{1-\theta} \frac{ \frac{i}{n}}{1-\frac{i}{n}} \right), \qquad
 k \in [n].
\end{equation}

Note that $(-\theta)^*_n > e^{-t}-1$ for large enough $n$, 
since $\theta<1-e^{-t}$ by \eqref{varphi_def} and \eqref{theta_def},
so  we can apply \eqref{f_n_lambda_perturb_compare} in \eqref{small_components_theta} and \eqref{big_components_theta} below. Now we bound the three terms of \eqref{split_components_theta}.

\medskip

{\bf First term:}
\begin{equation}
\label{small_components_theta}
 \mathbb{E} \left[\, X^*
   ; \, |\mathcal{C}| \in I_n \right]
\stackrel{  \eqref{I_J_K_intervals_def}, \eqref{f_n_lambda_perturb_compare} }{=}
 \mathbb{E} \left[\, X e^{\mathcal{O}\left( n^{-3/4} \right)} 
  \, ; \, |\mathcal{C}| \in I_n \right]
  \stackrel{ \eqref{I_J_K_intervals_def},  \eqref{f_n_minus_theta}   }{=}
  \mathbb{P} \left( |\mathcal{C}| \in I_n \right) +\mathcal{O}\left( \frac{1}{\sqrt{n}} \right).
\end{equation}

{\bf Second term} ($\mathbb{E} \left[\, X^*   ; \, |\mathcal{C}| \in J_n  \right]$):
\begin{equation}\label{x_tilde_bound}
e^{-(-\theta)^*_n t} \left( 1+\frac{ (-\theta)^*_n }{1-\frac{i}{n}} \right) 
\stackrel{ \eqref{theta_def}, \eqref{star_n}, \eqref{widetilde_lambda}   }{\leq}
\left( \frac{1+\widetilde{x}}{2} \right)^{-1} e^{-\widetilde{\lambda}^*_n t}
\left(1+\frac{\widetilde{\lambda}^*_n }{1-\frac{i}{n}}\right),
\qquad 1 \leq i \leq n^{3/4},
\end{equation}
\begin{multline}\label{medium_components_theta}
\mathbb{E} \left[\, X^*   ; \, |\mathcal{C}| \in J_n  \right]
\stackrel{\eqref{f_n_lambda_k}, \eqref{I_J_K_intervals_def}, \eqref{x_tilde_bound} }{\leq }
\mathbb{E} \left[ \left( \frac{1+\widetilde{x}}{2} \right)^{-|\mathcal{C}|}
 f_{n,t}( \widetilde{\lambda}^*_n  ,|\mathcal{C}| ) ; \,  |\mathcal{C}| \in J_n \right]
  \\ \stackrel{ \eqref{I_J_K_intervals_def} }{\leq}
 \left( \frac{1+\widetilde{x}}{2} \right)^{-n^{1/4}}
 \mathbb{E} \left[ 
 f_{n,t}( \widetilde{\lambda}^*_n  ,|\mathcal{C}| )  \right]
 \stackrel{ \eqref{mgf_ER_positive_f} }{=} \left( \frac{1+\widetilde{x}}{2} \right)^{-n^{1/4}} \left(1+ \widetilde{\lambda}^*_n \right) \stackrel{ \eqref{star_n}, \eqref{widetilde_lambda} }{ \leq } 
 e^{-\Omega(n^{1/4})}.
\end{multline}

{\bf Third term} ($ \mathbb{E}\left[\,  X^*
   ; \,  n^{3/4} \leq |\mathcal{C}|  \right] $):
\begin{equation}\label{less_than_one_in_third}
e^{-(-\theta)^*_n t} \left( 1+\frac{ (-\theta)^*_n }{1-\frac{i}{n}} \right) \leq 
e^{\theta t} (1-\theta) \stackrel{ \eqref{theta_def} }{=} 1 \quad \text{for any} \quad i \geq \lceil n^{3/4}\rceil,
\end{equation}
\begin{multline}
  \label{big_components_theta}
  \mathbb{E} \left[\,  X^*
   ; \,  n^{3/4} \leq |\mathcal{C}|  \right] \stackrel{ \eqref{f_n_lambda_k}, \eqref{less_than_one_in_third} }{\leq} 
  f_{n,t}( (- \theta )^*_n  ,  \lceil n^{3/4}\rceil )\stackrel{ \eqref{f_n_lambda_perturb_compare} }{ = }
  f_{n,t}( -\theta  ,  \lceil n^{3/4}\rceil )e^{\mathcal{O}\left( n^{-1/4} \right)} 
      \\ \stackrel{ \eqref{f_n_minus_theta} }{\leq}
   \exp\left( -\frac{\theta}{1-\theta} \sum_{i=0}^{ \lceil n^{3/4}\rceil -1} \frac{i}{n} \right)
   e^{\mathcal{O}\left( n^{-1/4} \right)}
  \leq e^{-\Omega(\sqrt{n})}.
\end{multline}

The statement of Lemma \ref{lemma_C_falls_in_first_interval_with_one_minus_theta} follows from
   \eqref{split_components_theta}, \eqref{small_components_theta},
  \eqref{medium_components_theta} and \eqref{big_components_theta}. 

\end{proof}

\subsection{Proof of Lemma \ref{lemma_clt_mgf_widetilde_I_n}}
\label{subsection_proof_of_lemma_clt_mgf_widetilde_I_n}

Before we outline the strategy of the proof of Lemma \ref{lemma_clt_mgf_widetilde_I_n} in the paragraph below \eqref{five_terms},
we need to introduce some notation. If we define
\begin{equation}\label{widetilde_alpha}
 \alpha^{**}_n := \frac{ \lfloor \sqrt{n} \alpha \rfloor }{\sqrt{n}}
\quad \text{ then } \quad
\left( \frac{\alpha}{\sqrt{n}} \right)_n^* \stackrel{ \eqref{star_n} }{=}
\frac{\alpha^{**}_n}{\sqrt{n}} 
\quad \text{ and } \quad
 |\alpha^{**}_n -\alpha| \leq \frac{1}{\sqrt{n}}.
\end{equation}

Let us abbreviate 
\[
Y^*=f_{n,t}\left(\frac{\alpha^{**}_n}{\sqrt{n}}, |\mathcal{C}| \right).
\]
Recall the definitions of the five intervals from \eqref{I_J_K_intervals_def} and \eqref{widetilde_I_K_intervals_def}. We have
\begin{multline}\label{five_terms}
\left(1+\frac{\alpha^{**}_n}{\sqrt{n}}\right)
\left(1- \mathbb{P}_{n+\lfloor \sqrt{n} \alpha \rfloor,p}[ \, |\mathcal{C}|>n \, ] \right) 
\stackrel{ \eqref{mgf_ER_positive_f} }{=}
\mathbb{E}[ Y^*; |\mathcal{C}| \in I_{n}   ] +
\mathbb{E}[ Y^*; |\mathcal{C}| \in J_{n}   ] \\+
\mathbb{E}[ Y^*; |\mathcal{C}| \in K_{n}   ] +
\mathbb{E}[ Y^*; |\mathcal{C}| \in \widetilde{I}_{n}   ] +
\mathbb{E}[ Y^*; |\mathcal{C}| \in \widetilde{K}_{n}   ].
\end{multline}
We will estimate the  five terms on the r.h.s.\ of \eqref{five_terms}. 
We will show that the terms corresponding to $I_n$ and $\widetilde{I}_n$ in \eqref{five_terms} approximate
the  terms corresponding to $I_n$ and $\widetilde{I}_n$ in \eqref{clt_mgf_widetilde_I_n} as
$n \to \infty$,
while the terms corresponding to $J_n$, $K_n$ and $\widetilde{K}_n$ in \eqref{five_terms} vanish as $n \to \infty$.

\begin{proof}[Proof of Lemma \ref{lemma_clt_mgf_widetilde_I_n}]
Before we start estimating the five terms of \eqref{five_terms}, we note that if $k \in I_{n} \cup  J_{n} \cup K_n \cup \widetilde{I}_{n}$ then we can use Taylor
expansion of $\ln(1+x)$ to obtain for any $\alpha \in \R$ the formula
\begin{multline}\label{f_n_calculation}
f_{n,t}\left( \frac{\alpha}{\sqrt{n}},k \right)
\stackrel{ \eqref{f_n_lambda_k} }{=}
\exp\left( -\frac{\alpha}{\sqrt{n}}kt + 
\sum_{i=0}^{k-1} 
\ln \left( 1+ \frac{ \frac{\alpha}{\sqrt{n} } }{1-\frac{i}{n}} \right)\right)
\\=
\exp\left( 
  -\frac{\alpha}{\sqrt{n}} kt + \sum_{i=0}^{k-1} \frac{\alpha}{\sqrt{n}} \frac{1}{1-\frac{i}{n}} \
-\frac{1}{2} \sum_{i=0}^{k-1} \frac{\alpha^2}{n} \frac{1}{(1-\frac{i}{n})^2} +\mathcal{O}\left( \frac{1}{\sqrt{n}} \right)  \right)\\=
\exp\left( \frac{\alpha}{\sqrt{n}}\left( -kt + n \int_{0}^{\frac{k}{n}}\frac{1}{1-x}\, \mathrm{d}x \right) 
-\frac{\alpha^2}{2} \int_0^{\frac{k}{n}} \frac{1}{(1-x)^2} \, \mathrm{d}x +
\mathcal{O}\left( \frac{1}{\sqrt{n}} \right)  \right)\\
\stackrel{ \eqref{varphi_def}  }{=}
\exp\left( \alpha\sqrt{n}\varphi\left( \frac{k}{n} \right) 
-\frac{\alpha^2}{2} \frac{\frac{k}{n}}{1-\frac{k}{n}} +
\mathcal{O}\left( \frac{1}{\sqrt{n}} \right)  \right).
\end{multline}

Now we can estimate the five terms on the r.h.s.\ of \eqref{five_terms}.

{\bf First term:}
\begin{equation}\label{small_alpha}
\mathbb{E}[ Y^*; |\mathcal{C}| \in I_{n}   ] \stackrel{ \eqref{varphi_def}, \eqref{I_J_K_intervals_def}, \eqref{f_n_calculation} }{=} \mathbb{P} \left( |\mathcal{C}| \in I_n \right) +\mathcal{O}\left( n^{-1/4} \right).
\end{equation}

{\bf Second term:}
The bound 
\begin{equation}\label{medium_alpha}
\mathbb{E}[ Y^*; |\mathcal{C}| \in J_{n}   ] \leq  e^{-\Omega(n^{1/4})}
\end{equation}
 can be deduced analogously to \eqref{medium_components_theta} using that for large enough $n$ we have
\begin{equation}
e^{-\frac{\alpha^{**}_n}{\sqrt{n}} t} \left( 1+\frac{ \frac{\alpha^{**}_n}{\sqrt{n}} }{1-\frac{i}{n}} \right) 
\stackrel{  \eqref{star_n} , \eqref{widetilde_lambda} }{\leq}
\left( \frac{1+\widetilde{x}}{2} \right)^{-1} e^{-\widetilde{\lambda}^*_n t}
\left(1+\frac{\widetilde{\lambda}^*_n }{1-\frac{i}{n}}\right),
\qquad 1 \leq i \leq n^{3/4}.
\end{equation}

{\bf Third term} ($ \mathbb{E}[ Y^*; |\mathcal{C}| \in K_{n}   ] $): We note 
\begin{equation}\label{K_n_bound_Y_start}
\frac{ f_{n,t}\left( \frac{\alpha^{**}_n}{\sqrt{n}}, k \right)  }
{  f_{n,t}\left( \frac{(\alpha-1)^{**}_n}{\sqrt{n}}, k \right) } 
\stackrel{  \eqref{f_n_calculation} }{ = }
\exp\left( \sqrt{n} \varphi\left( \frac{k}{n} \right) +\mathcal{O}(1) \right)
\stackrel{ \eqref{varphi_def}, \eqref{theta_def}, \eqref{I_J_K_intervals_def} }{\leq} e^{-\Omega(n^{1/8})},
\quad k \in K_n,
\end{equation}
\begin{equation}\label{K_n_expect_bound_middle}
\mathbb{E}[ Y^*; |\mathcal{C}| \in K_{n}   ] \stackrel{ \eqref{K_n_bound_Y_start} }{ \leq }
e^{-\Omega(n^{1/8})} \mathbb{E}\left[ f_{n,t}\left(\frac{(\alpha-1)^{**}_n}{\sqrt{n}}, |\mathcal{C}| \right)  ; |\mathcal{C}| \in K_{n} \right] \stackrel{ \eqref{mgf_ER_positive_f} }{\leq}
2 e^{-\Omega(n^{1/8})}. 
\end{equation}

{\bf Fourth term} ($ \mathbb{E}[ Y^*; |\mathcal{C}| \in \widetilde{I}_{n} ] $): If $x \in [-n^{1/8},n^{1/8}]$,
i.e., if $k=\lfloor \theta n +x \sqrt{n} \rfloor \in \widetilde{I}_{n}$   then
\begin{equation}\label{giant_square_root_window}
f_{n,t} \left(\frac{\alpha}{\sqrt{n}}, k \right) 
\stackrel{  \eqref{theta_def}, \eqref{f_n_calculation}  }{=}
\exp\left( \alpha \varphi'(\theta) x  - \frac{\alpha^2}{2} \frac{\theta}{1-\theta} 
+ \mathcal{O}\left( n^{-1/4 } \right)
\right),
\end{equation}
\begin{equation}\label{f_n_t_widetilde_compare_to_mgf_with_error}
f_{n,t} \left(\frac{\alpha^{**}_n}{\sqrt{n}}, k \right)
\stackrel{ \eqref{widetilde_alpha}, \eqref{giant_square_root_window} }{=}
\exp\left( \alpha \varphi'(\theta) x  - \frac{\alpha^2}{2} \frac{\theta}{1-\theta} 
 \right) +\mathcal{O}\left( n^{-1/4}  f_{n,t} \left(\frac{\alpha^{**}_n}{\sqrt{n}}, k \right) \right),
\end{equation}
\begin{equation}\label{I_n_mgf_expect_estimate}
\mathbb{E}[ Y^*; |\mathcal{C}| \in \widetilde{I}_{n} ]
\stackrel{  \eqref{mgf_ER_positive_f}, \eqref{widetilde_I_K_intervals_def}, \eqref{f_n_t_widetilde_compare_to_mgf_with_error} }{=}
\mathbb{E}\left[ 
\exp\left( \alpha \varphi'(\theta) \frac{ |\mathcal{C}|-\theta n }{\sqrt{n}}  - \frac{\alpha^2}{2} \frac{\theta}{1-\theta} \right);\; |\mathcal{C}| \in \widetilde{I}_{n} \right]+\mathcal{O}\left( n^{-1/4 } \right).
\end{equation}

{\bf Fifth term} ($ \mathbb{E}[ Y^*; |\mathcal{C}| \in \widetilde{K}_n   ] $):
 We observe that 
\begin{equation}\label{K_n_tilde_bound_left_endopint}
\frac{ f_{n,t}\left( \frac{\alpha^{**}_n}{\sqrt{n}}, \lfloor \theta n+ n^{1/8} \sqrt{n} \rfloor \right)  }{  f_{n,t}\left( \frac{(\alpha+1)^{**}_n}{\sqrt{n}}, \lfloor \theta n+ n^{1/8} \sqrt{n} \rfloor \right) } \stackrel{ \eqref{giant_square_root_window} }{=}
\exp\left( -  \varphi'(\theta) n^{1/8} + \mathcal{O}(1) \right)
 \stackrel{ \eqref{varphi_prime_theta} }{\leq} e^{-\Omega(n^{1/8})},
\end{equation}
\begin{equation}\label{K_n_tilde_bound_increment}
\exp\left(-\frac{\alpha^{**}_n}{\sqrt{n}} t\right) \left( 1+\frac{ \frac{\alpha^{**}_n}{\sqrt{n}} }{1-\frac{i}{n}} \right) \stackrel{ \eqref{varphi_prime_theta} }{\leq}
\exp\left(-\frac{(\alpha+1)^{**}_n}{\sqrt{n}} t\right) \left( 1+\frac{ \frac{(\alpha+1)^{**}_n}{\sqrt{n}} }{1-\frac{i}{n}} \right), \quad  \lfloor \theta n+ n^{1/8} \sqrt{n} \rfloor \leq i,
\end{equation}
\begin{equation}\label{K_n_tilde_bound}
f_{n,t} \left(\frac{\alpha^{**}_n}{\sqrt{n}}, k \right)\stackrel{ \eqref{f_n_lambda_k}, \eqref{K_n_tilde_bound_left_endopint}, \eqref{K_n_tilde_bound_increment} }{\leq}
e^{-\Omega(n^{1/8})} f_{n,t} \left(\frac{(\alpha+1)^{**}_n}{\sqrt{n}}, k \right), \quad
\lfloor \theta n+ n^{1/8} \sqrt{n} \rfloor \leq k,
\end{equation}
\begin{equation}\label{K_n_tilde_expect_bound}
\mathbb{E}[ Y^*; |\mathcal{C}| \in \widetilde{K}_n   ] 
\stackrel{ \eqref{widetilde_I_K_intervals_def}, \eqref{K_n_tilde_bound} }{ \leq }
e^{-\Omega(n^{1/8})} \mathbb{E}\left[ f_{n,t}\left(\frac{(\alpha+1)^{**}_n}{\sqrt{n}}, |\mathcal{C}| \right)  ; |\mathcal{C}| \in \widetilde{K}_{n} \right] \stackrel{ \eqref{mgf_ER_positive_f} }{\leq}
2 e^{-\Omega(n^{1/8})}. 
\end{equation}

Finally, the proof of the fact that the error term $\mathbb{P}_{n+\lfloor \sqrt{n} \alpha \rfloor,p}[ \, |\mathcal{C}|>n \, ]$ that appears on the l.h.s.\
of \eqref{five_terms} goes to zero as $n \to \infty$ is analogous to the $\alpha=0$ case of
\eqref{K_n_tilde_expect_bound}.
The statement of  Lemma \ref{lemma_clt_mgf_widetilde_I_n} follows from
\eqref{five_terms}, \eqref{small_alpha}, \eqref{medium_alpha}, \eqref{K_n_expect_bound_middle},
\eqref{I_n_mgf_expect_estimate} and \eqref{K_n_tilde_expect_bound}.
 
 \end{proof}

\medskip 
 
 {\bf Acknowledgements:} The author thanks James Martin, Dominic Yeo and B\'alint T\'oth for valuable discussions and an anonymous referee for thoroughly reading the manuscript and for suggesting to use 
 \cite[Lemma 2.1]{remco_local_limit} in the proof of Theorem \ref{thm_CLT_giant}.
   This work is partially supported by OTKA (Hungarian
National Research Fund) grant K109684, the Postdoctoral Fellowship of
NKFI (National Research, Development and Innovation Office) and the Bolyai Research Scholarship
of the Hungarian Academy of Sciences.

\end{document}